\pdfoutput=1  % force to compile with pdflatex  
\documentclass[10pt]{amsart}
\usepackage{amscd}
\usepackage{amssymb}
\usepackage[all]{xy}
\usepackage{lmodern}
\usepackage{graphicx}
\title{Flatness and Completion Revisited}
\date{6 September 2017}
\usepackage{hyperref}
\hypersetup{colorlinks=false}
\author{Amnon Yekutieli}
\address{Department of  Mathematics,
Ben Gurion University, Be'er Sheva 84105, Israel}
\email{amyekut@math.bgu.ac.il}

\thanks{
Primary 13J10; Secondary 13B35, 13E05, 13C10, 13C11, 13D07.}

\newtheorem{thm}[equation]{Theorem}
\newtheorem{cor}[equation]{Corollary}
\newtheorem{prop}[equation]{Proposition}
\newtheorem{lem}[equation]{Lemma}
\theoremstyle{definition}
\newtheorem{dfn}[equation]{Definition}
\newtheorem{rem}[equation]{Remark}
\newtheorem{exa}[equation]{Example}

\newtheorem{que}[equation]{Question}

\numberwithin{equation}{section}
\newcommand{\xar}{\xrightarrow}
\newcommand{\opn}{\operatorname}
\newcommand{\cat}[1]{\operatorname{\mathsf{#1}}}

\newcommand{\cd}{\,{\cdot}\,}

\newcommand{\rmitem}[1]{\item[\text{\textup{(#1)}}]}
\newcommand{\mfrak}[1]{\mathfrak{#1}}

\newcommand{\mrm}[1]{\mathrm{#1}}

\newcommand{\Ga}{\Gamma}
\newcommand{\La}{\Lambda}

\newcommand{\de}{\delta}

\newcommand{\Om}{\Omega}

\renewcommand{\a}{\mfrak{a}}
\renewcommand{\b}{\mfrak{b}}

\newcommand{\K}{\mathbb{K}}

\newcommand{\Z}{\mathbb{Z}}
\newcommand{\N}{\mathbb{N}}

\newcommand{\tup}[1]{\textup{#1}}
\newcommand{\bsym}[1]{\boldsymbol{#1}}

\newcommand{\ot}{\otimes}

\newcommand{\what}[1]{\widehat{#1}}
\newcommand{\mhat}[1]{\skew{3.0}{\what}{#1}}

\renewcommand{\d}{\mathrm{d}}

\newcommand{\lb}{\linebreak}

\begin{document}

\begin{abstract}
We continue investigating the interaction between flatness 
and $\a$-adic completion for infinitely generated $A$-modules.
Here $A$ is a commutative ring and $\a$ is a finitely generated ideal in it. We 
introduce the concept of {\em $\a$-adic flatness}, which is weaker 
than flatness. We prove that $\a$-adic flatness is preserved under completion 
when the ideal $\a$ is {\em weakly proregular}. We also prove that when $A$ is 
noetherian, $\a$-adic flatness coincides with flatness (for complete modules). 
An example is worked out of a non-noetherian ring $A$, with a  weakly 
proregular ideal $\a$, for which the completion $\what{A}$ is not flat. We also 
study {\em $\a$-adic systems}, and prove that if the ideal $\a$ is finitely 
generated, then the limit of every $\a$-adic system is a complete module.
\end{abstract}

\maketitle
\tableofcontents

%\cleardoublepage
\setcounter{section}{-1}
\section{Introduction}

In this paper we continue investigating the interaction between flatness 
and adic completion for infinitely generated modules over a commutative 
ring, that we had started in \cite{Ye1}. 

Let $A$ be a commutative ring, and let $\a$ be an ideal in it. For each $k \in 
\N$ we define the quotient ring
$A_k := A / \a^{k + 1}$. 
The collection of rings $\{ A_k \}_{k \in \N}$ is an inverse system,
and the $\a$-adic completion of $A$ is the commutative ring 
$\what{A} := \lim_{\leftarrow k} \, A_k$. 
There is a canonical ring homomorphism $A \to \what{A}$. 

An {\em $\a$-adic system} of $A$-modules is an inverse system
$\{ M_k \}_{k \in \N}$, where each $M_k$ is an $A_k$-module, and the induced 
homomorphisms 
\[ A_k \ot_{A_{k + 1}} M_{k + 1} \to M_k \]
are all bijective. The {\em limit} of the system is the $A$-module 
$\what{M} := \lim_{\leftarrow k} \, M_k$.

Given an $A$-module $M$, we let 
$M_k := A_k \ot_A M$.
The collection of $A$-modules $\{ M_k \}_{k \in \N}$ is an inverse system,
that we call the {\em $\a$-adic system induced by $M$}. The limit $\what{M}$ of 
the induced system is the {\em $\a$-adic completion of $M$}.
There is a canonical $A$-module homomorphism 
$\tau_M : M \to \what{M}$, and $M$ is called {\em $\a$-adically complete} if 
$\tau_M$ is an isomorphism. (Some texts, mostly older ones, would say that $M$ 
is complete and separated.)

Our initial motivation was to prove the following theorem, that we consider 
important.

\begin{thm} \label{thm:165}
If $A$ is a noetherian commutative ring, $\a$ is an ideal in $A$, and $M$ is a 
flat $A$-module, then the $\a$-adic completion $\what{M}$ is a flat $A$-module.
\end{thm}

Of course the module $M$ has to be {\em infinitely generated} for this 
statement to be interesting.  

We were under the impression that this was an open question, solved only in 
special cases (cf.\ \cite{En}, \cite{BS},  \cite{Sct}, and
\cite[Theorems 4.2 and 3.4(2)]{Ye1}). This impression was 
based on a literature search and email correspondence with several experts. 
A few months ago we found a relatively simple proof of this result (it is now 
the proof of Theorem \ref{thm:262}(2) below). 

However, after email correspondence with a few more experts, to whom we showed 
a preliminary version of this paper, we learned that:
\begin{enumerate}
\item There is already a proof of Theorem \ref{thm:165}. It is  
\cite[Lemma 0AGW]{SP}, presumably due to de Jong, from around 2013. 

\item There is an elementary proof of Theorem \ref{thm:165}, 
indicated to us by Gabber and Ramero, along the lines of \cite[Lemma 7.1.6]{GR}.
\end{enumerate}

Thus we turned our attention to a few theorems of secondary importance, that 
are totally new, and to which our methods could be applied. This is the content 
of the present paper.

Before going on, let us make several general remarks about $\a$-adic 
completion and $\a$-torsion, that are not widely known. For an $A$-module $M$, 
let us write $\La_{\a}(M) := \what{M}$, and let $\Ga_{\a}(M)$ be the 
$\a$-torsion submodule of $M$. So $\La_{\a}$ and $\Ga_{\a}$ are $A$-linear 
functors from the category $\cat{Mod} A$ of $A$-modules to itself. 

If the ideal $\a$ is finitely generated, then for every $A$-module $M$ its 
$\a$-adic completion $\what{M}$ is $\a$-adically complete. 
Thus the functor $\La_{\a}$ is idempotent, in the sense that 
$\La_{\a} \circ \La_{\a} \cong \La_{\a}$.
There are counterexamples to this idempotence when the ideal $\a$ is not 
finitely generated. See \cite[Corollary 3.6 and Example 1.8]{Ye1}.

Even when the ring $A$ is noetherian (which of course forces the ideal $\a$ to 
be finitely generated), the functor $\La_{\a}$ is not left-exact nor 
right-exact. It is true that when $A$ is noetherian, the functor $\La_{\a}$ is 
exact when restricted to the category $\cat{Mod}_{\mrm{f}} A$ of finitely 
generated $A$-modules.

The functor $\Ga_{\a}$ is always right exact and idempotent. 

The first main result of our paper is on $\a$-adic systems. It is a vast 
generalization of the classical case, in which $A$ is noetherian and each $M_k$ 
is a finitely generated $A_k$-module. 

\begin{thm} \label{thm:180}
Let $A$ be a commutative ring, let $\a$ be a finitely generated ideal in $A$, 
and let $\{ M_k \}_{k \in \N}$ be an $\a$-adic system of 
$A$-modules, with limit $\what{M}$. Then\tup{:}
\begin{enumerate}
\item The $A$-module $\what{M}$ is $\a$-adically complete.

\item For every $k \geq 0$ the canonical homomorphism 
$A_k \ot_A \what{M} \to M_k$ is bijective.
\end{enumerate}
\end{thm}

This theorem is repeated as Theorem \ref{thm:230} in Section \ref{sec:adic-sys} 
and proved there. An immediate consequence of Theorem \ref{thm:180} is that 
when the ideal $\a$ is finitely generated, every $\a$-adic system 
$\{ M_k \}_{k \in \N}$ is induced from a module.
See Corollaries \ref{cor:333} and \ref{cor:263}.

Now we turn to flatness. 
An $A$-module $M$ is said to be {\em $\a$-adically 
flat} if \lb $\opn{Tor}^A_i(N, M) = 0$ for every $i > 0$ and every 
$\a$-torsion $A$-module $N$. Clearly flatness implies $\a$-adic flatness.  

Here is a useful characterization of $\a$-adically flat modules. 

\begin{thm} \label{thm:261}
Let $A$ be a commutative ring, let $\a$ be an ideal in $A$, and 
let $M$ be an $A$-module. The following conditions are equivalent\tup{:}
\begin{enumerate}
\rmitem{i} The $A$-module $M$ is $\a$-adically flat.

\rmitem{ii} For every $i > 0$ and $k \geq 0$ the module 
$\opn{Tor}^{A}_i(A_{k}, M)$ vanishes, and   
$A_{k} \ot_A M$ is a flat $A_k$-module.  

\rmitem{iii} For every $i > 0$ the module 
$\opn{Tor}^{A}_i(A_{0}, M)$ vanishes, and   
$A_{0} \ot_A M$ is a flat $A_0$-module. 
\end{enumerate}
\end{thm}

This is repeated as Theorem \ref{thm:168} in the body of the paper. 
Observe that there are no finiteness conditions on $A$, $\a$ or $M$. 
The result is similar to \cite[Lemma 051C]{SP} -- see Remark \ref{rem:310} 
for a comparison. The proof of Theorem \ref{thm:261} relies on some basic 
properties of the derived tensor product. 

Suppose $\bsym{a} = (a_1, \ldots, a_n)$ is a finite sequence of elements in 
$A$. The sequence $\bsym{a}$ is called {\em weakly proregular} if it satisfies 
a 
rather complicated condition, involving the Koszul complexes associated to 
powers of $\bsym{a}$. The definition is recalled in Section 
\ref{sec:WPR}. The ideal $\a$ is called weakly proregular if it is 
generated by some weakly proregular sequence. It is known that when $A$ 
is noetherian, every ideal in it is weakly proregular.
But there are fairly natural examples of weakly proregular ideals 
in non-noetherian rings (see Theorem \ref{thm:205} for such an example). 
It is now understood (see \cite{PSY1} and \cite{PSY2}) that weak proregularity 
of the ideal $\a$ is a necessary and sufficient condition for the derived 
functors $\mrm{L} \La_{\a}$ and $\mrm{R} \Ga_{\a}$ to have ``good behavior''.

As the next results show, adic flatness belongs to the ``twilight 
zone'' between the weakly proregular case and the noetherian case. 
Adic flatness comes up in the weakly proregular situation:

\begin{thm} \label{thm:280}
Let $A$ be a commutative ring, let $\a$ be a weakly proregular ideal in $A$, 
and let $M$ be an $\a$-adically flat $A$-module, with $\a$-adic completion 
$\what{M}$. Then the $A$-module $\what{M}$ is $\a$-adically flat. 
\end{thm}

But then the two notions of flatness merge in the noetherian case:

\begin{thm} \label{thm:281}
Let $A$ be a noetherian commutative ring, let $\a$ be an ideal in $A$, 
and let $\what{M}$ be an $\a$-adically flat $\a$-adically complete $A$-module. 
Then $\what{M}$ is a flat $A$-module. 
\end{thm}

Theorem \ref{thm:280} is repeated as Theorem \ref{thm:200} in the body of the 
paper. The proof uses derived categories and the MGM equivalence. 
On the other hand, Theorem \ref{thm:281} (repeated as Corollary \ref{cor:320}), 
is a consequence of item (2) of the following theorem, combined with Theorem 
\ref{thm:261}. 

An $\a$-adic system $\{ M_k \}_{k \in \N}$ is called {\em flat} if each $M_k$ 
is a flat $A_k$-module.  

\begin{thm} \label{thm:262}
Let $A$ be a commutative ring, let $\a$ be an ideal in $A$, and 
let $\{ M_k \}_{k \in \N}$ be a flat  $\a$-adic system, with limit 
$\what{M}$.

\begin{enumerate} 
\item If the ideal $\a$ is weakly proregular, then $\what{M}$ is an 
$\a$-adically flat $A$-module. 
\item If the ring $A$ is noetherian, then $\what{M}$ is a flat $A$-module. 
\end{enumerate}
\end{thm}

Parts (1) and (2) of the theorem are repeated as Theorems \ref{thm:290} and 
\ref{thm:295} respectively. The proofs use {\em free resolutions of 
$\a$-adic systems}, and properties of 
{\em modules of decaying functions} that were established in \cite{Ye1}.

The reader might wonder whether the distinction between flatness and 
adic flatness is genuine. The answer is that these notions are indeed distinct. 
In Theorem \ref{thm:205} we give an example of a ring $A$ and a weakly 
proregular ideal $\a \subseteq A$, such that 
the $\a$-adic completion $\what{A}$ -- which is an $\a$-adically flat 
$A$-module by Theorem \ref{thm:280} -- {\em is not flat over $A$}. 
This example is not exotic at all. The ring is 
\[ A := \K[[t_1]] \ot_{\K} \K[[t_2]] , \]
where $\K$ is a field of characteristic $0$; and the ideal is 
$\a := (t_1, t_2)$. 

In Sections \ref{sec:adic-flatness} and \ref{sec:WPR} we use {\em derived 
category} methods. All the necessary background on these methods can be found in 
the new book \cite{Ye3}. Perhaps some of the results in Section 
\ref{sec:adic-flatness} could be proved using only classical homological 
algebra (i.e.\ Ext, Tor and related spectral sequences); but even if so, the 
proofs would very likely be much longer and more difficult to understand. 
The proofs in Section \ref{sec:WPR}, that involve the MGM 
Equivalence, definitely require the use of derived categories. 

The content of this paper belongs to commutative algebra. However, it is 
expected to have applications in {\em noncommutative 
ring theory} and {\em representation theory}. Indeed, our paper \cite{Ye1} was 
written as part of our project on {\em deformation quantization}, for which we 
needed sharper results on $\a$-adically complete flat noncommutative central 
$A$-rings (and sheaves of this type). See the paper \cite{Ye2} and its 
references.

\medskip \noindent
{\bf Acknowledgments}. 
Thanks to Liran Shaul, Sean Sather-Wagstaff, Asaf Yekutieli, Steven Kleiman, 
Brian Conrad, Ofer Gabber, Lorenzo Ramero, Pierre Deligne, Johan de Jong,  
Ilya Tyomkin and Leonid Positselski for helpful discussions. We also wish to 
thank the anonymous referee, for reading the paper carefully and suggesting 
several improvements.

%\cleardoublepage
\section{\texorpdfstring{$\a$}{a}-Adic Systems} 
\label{sec:adic-sys}

Throughout the paper $A$ is a commutative ring, and $\a$ is an ideal in it. 
Recall the inverse system of rings 
$\{ A_k \}_{k \in \N}$, where $A_k = A / \a^{k + 1}$. 

\begin{dfn} \label{dfn:15}
An {\em $\a$-adic system of $A$-modules} is data 
\[ \bsym{M} = \bigl( \{ M_k \}_{k \in \N}, \, \{ \nu_{k} \}_{k \in \N} \bigr) , 
\]
where $M_k$ is an $A_k$-module, and 
$\nu_{k} : M_{k + 1} \to M_k$
is an $A$-module homomorphism, called the {\em $k$-th transition}.
The condition is that for every $k$, the homomorphism 
\[ A_k \ot_{A_{k + 1}} M_{k + 1} \to M_k  \]
induced by the transition $\nu_k$ is bijective. 
\end{dfn}

Usually the collection of transitions $\{ \nu_k \}_{k \in \N}$
will remain implicit.

\begin{dfn} \label{dfn:22}
Let $\bsym{M} = \{ M_k \}_{k \in \N}$ and 
$\bsym{N} = \{ N_k \}_{k \in \N}$ be $\a$-adic systems. A {\em morphism of 
$\a$-adic systems}
$\bsym{\phi}: \bsym{M} \to \bsym{N}$
is a collection of $A$-module homomorphisms 
$\phi_k : M_k \to N_k$, that commute with the transitions in 
$\bsym{M}$ and $\bsym{N}$. 
In this way the $\a$-adic systems become a category, that we denote by 
$\cat{Sys} (A, \a)$.
\end{dfn}

\begin{dfn} \label{dfn:200}
Let $\bsym{M} = \{ M_k \}_{k \in \N}$ be an $\a$-adic system of $A$-modules.
The {\em limit} of $\bsym{M}$ is the $A$-module 
$\what{M} := \lim_{\leftarrow k} \, M_k$.
\end{dfn}

\begin{exa} \label{exa:126}
Consider an $A$-module $M$. There is an {\em induced $\a$-adic system}
$\bsym{M} = \lb \{ M_k \}_{k \in \N}$, 
defined by $M_k := A_k \ot_A M$,
with the obvious transitions. 
The limit of the system $\bsym{M}$ is the $\a$-adic completion of the module 
$M$. 
\end{exa}

Two obvious questions come to mind:

\begin{que} \label{que:265}
Is the limit $\what{M}$ of every $\a$-adic system $\{ M_k \}_{k \in \N}$ an 
$\a$-adically complete module?
\end{que}

\begin{que} \label{que:266}
Is every $\a$-adic system $\{ M_k \}_{k \in \N}$ induced by a module $M$~?
\end{que}

Suppose $A$ is noetherian, with $\a$-adic  completion $\what{A}$.
Let $\{ M_k \}_{k \in \N}$ be an $\a$-adic system such that each 
$M_k$ is a finitely generated $A_k$-module. In this case it is 
well known that the limit $\what{M}$ is a finitely generated $\what{A}$-module, 
and hence it is complete. Moreover, the system $\{ M_k \}_{k \in \N}$
is induced from $\what{M}$. See \cite{AM}, \cite{Ma} or \cite{CA}. We see that 
both questions have positive answers in this case. 

At the opposite end, here is a counterexample to Question \ref{que:265},
when the ideal $\a$ is not finitely generated. 

\begin{exa} \label{exa:260}
Take the ring  
$A = \K[t_0, t_1, \ldots ]$ 
of polynomials in countably many variables over a field $\K$. 
Let $\a$ be the maximal ideal generated by the variables. 
Consider the $A$-module $M = A$. Let 
$\bsym{M} = \{ M_k \}_{k \in \N}$ be the 
$\a$-adic system induced by $M$ (so actually $M_k = A_k$), and let 
$\what{M}$ be the limit of this system.
Then $\what{M}$ is the $\a$-adic completion of the module $M$. As shown in 
\cite[Example 1.8]{Ye1}, the $A$-module $\what{M}$ is not $\a$-adically 
complete. It is also shown there that the canonical homomorphisms 
$A_k \ot_A \what{M} \to M_k$ are not bijective. 
\end{exa}

We could not find a counterexample to Question \ref{que:266}.

As the next theorem shows, both questions have positive answers when the 
ideal $\a$ is finitely generated. 

\begin{thm} \label{thm:230}
Let $A$ be a commutative ring, let $\a$ be a finitely generated ideal in $A$, 
and let $\{ M_k \}_{k \in \N}$ be an $\a$-adic system of 
$A$-modules, with limit $\what{M}$. Then\tup{:}
\begin{enumerate}
\item The $A$-module $\what{M}$ is $\a$-adically complete.

\item For every $k \geq 0$ the canonical homomorphism 
$A_k \ot_A \what{M} \to M_k$ is bijective.
\end{enumerate} 
\end{thm}

\begin{proof}
(1) For each $k$ let 
$\what{M}_k := A_k \ot_A \what{M}$. 
In this way we obtain a second $\a$-adic system 
$\{ \what{M}_k  \}_{k \in \N}$. Its limit is the module 
$\mhat{\what{M}}$, which is the $\a$-adic completion of the module 
$\what{M}$. According to \cite[Corollary 3.6]{Ye1} the $A$-module 
$\mhat{\what{M}}$ is $\a$-adically complete.

For each $k$ there is the canonical homomorphism 
$\pi_k : \what{M} \to M_k$ of the limit, and the canonical homomorphism 
$\tau_k : \what{M} \to \what{M}_k$ 
induced by $A \to A_k$. There is also the canonical homomorphism
$\phi_k : \what{M}_k \to M_k$ from item (2). 
They form a commutative diagram 
\[ \UseTips \xymatrix @C=6ex @R=6ex {
\what{M}
\ar[r]^{\tau_k}
\ar@(ur,ul)[rr]^{\pi_k}
&
\what{M}_k
\ar[r]^{\phi_k}
& 
M_k
} \]
Passing to the limit we obtain the commutative diagram 
\[ \UseTips \xymatrix @C=6ex @R=6ex {
\what{M}
\ar[r]^{\tau}
\ar@(ur,ul)[rr]^{=}
&
\mhat{\what{M}}
\ar[r]^{\phi}
& 
\what{M}
} \]
%Here $\phi := \lim_{\leftarrow k} \, \phi_k$. 
We see that $\tau = \tau_{\what{M}}$ is a split injection, and thus $\what{M}$ 
is a direct summand of $\mhat{\what{M}}$. But a direct summand of an  
$\a$-adically complete module is itself $\a$-adically complete; 
so we conclude that $\what{M}$ is $\a$-adically complete. 
This means that the homomorphism 
$\tau_{\what{M}} : \what{M} \to \what{\what{M}}$ 
is bijective. But then $\phi$ is also bijective. 

\medskip \noindent 
(2) Continuing with the same notation, for each $k$ let 
$L_k := \opn{Ker}(\phi_k)$. We get an inverse system of exact sequences 
\begin{equation} \label{eqn:230}
 0 \to L_k \to \what{M}_k \xar{\phi_k} M_k \to  0 . 
\end{equation}
If we apply the functor $A_k \ot_{A_{k + 1}} -$ to the sequence with index
$k + 1$, we get an exact sequence isomorphic to 
\[  A_k \ot_{A_{k + 1}} L_{k + 1} \to \what{M}_k \xar{\phi_k} M_k \to  0 . \]
We see that the homomorphism $L_{k + 1} \to L_k$ is surjective.
Define $\what{L} := \lim_{\leftarrow k} \, L_k$.
Passing to the limit in (\ref{eqn:230}), the Mittag-Leffler argument says that 
\[ 0 \to \what{L} \to \mhat{\what{M}} \xar{\phi} \what{M} \to 0  \]
is an exact sequence. We already know that $\phi$ is bijective, and therefore 
$\what{L} = 0$. Finally, because the canonical homomorphism 
$\what{L} \to L_k$ is surjective, it follows that $L_k = 0$ and that $\phi_k$ is 
bijective. 
\end{proof}

\begin{cor} \label{cor:333}
If the ideal $\a$ is finitely generated, then every $\a$-adic system \lb 
$\{ M_k \}_{k \in \N}$ is induced from a module.
\end{cor}

\begin{proof}
By item (2) of the theorem, the system $\{ M_k \}_{k \in \N}$ is induced from 
the module $\what{M}$.
\end{proof}

Let $\cat{Mod}_{\a\tup{-com}} A$ be the category of $\a$-adically complete 
modules (a full subcategory of $\cat{Mod} A$). 

\begin{cor} \label{cor:263}
If the ideal $\a$ is finitely generated, then the functor sending a module $M$ 
to the induced $\a$-adic system $\{ M_k \}_{k \in \N}$ is an equivalence of 
categories
\[ \cat{Mod}_{\a\tup{-com}} A \to \cat{Sys} (A, \a) . \]
Its quasi-inverse is the limit functor. 
\end{cor}

\begin{proof}
Let $M$ be a complete module, and let $\bsym{M}$ be the induced $\a$-adic 
system. As mentioned in Example \ref{exa:126}, the limit of $\bsym{M}$
is the $\a$-adic completion $\what{M}$ of $M$. But $M \cong \what{M}$.

Conversely, if we start with an $\a$-adic system $\bsym{M}$, 
with limit $\what{M}$, then by Theorem \ref{thm:230} the module 
$\what{M}$ is complete, and the $\a$-adic system induced by $\what{M}$ is 
isomorphic to $\bsym{M}$. 
\end{proof}

%\cleardoublepage
\section{Free Resolutions of \texorpdfstring{$\a$}{a}-Adic Systems} 
\label{sec:res-adic-sys}

As before, $A$ is a commutative ring and $\a$ is an ideal in it. 

Let us recall some definitions 
from our paper \cite{Ye1}. Let $Z$ be a set and let $N$ be an 
$A$-module. The set of all functions $f : Z \to N$ is denoted by 
$\opn{F}(Z, N)$. 
It is an $A$-module in the obvious way. The {\em support} of a  function 
$f : Z \to N$ is the set 
$\{ z \in Z \mid f(z) \neq 0 \}$.  
We denote by 
$\opn{F}_{\mrm{fin}}(Z, N)$ 
the submodule of $\opn{F}(Z, N)$ consisting of finite support functions.
For $N = A$, the $A$-module 
$\opn{F}_{\mrm{fin}}(Z, A)$
is free, with basis the collection 
$\{ \de_z \}_{z \in Z}$ of delta functions.
For every $A$-module $N$ there are canonical isomorphisms
\begin{equation} \label{eqn:73}
\opn{F}_{\mrm{fin}}(Z, A) \ot_A N \cong 
\opn{F}_{\mrm{fin}}(Z, N) \cong 
\bigoplus_{z \in Z} \, N . 
\end{equation}

Given a set $Z$, we have an $\a$-adic system of $A$-modules 
$\bigl\{ \opn{F}_{\mrm{fin}}(Z, A_k) \bigr\}_{k \in \N}$. 
See Definition \ref{dfn:22} regarding morphisms of $\a$-adic systems.  

\begin{dfn} \label{dfn:20}
An $\a$-adic system of $A$-modules $\bsym{P}$
is said to be a {\em free $\a$-adic system} if there is an 
isomorphism of $\a$-adic systems 
\[ \bsym{P} \cong \bigl\{ \opn{F}_{\mrm{fin}}(Z, A_k) \bigr\}_{k \in \N}  \]
for some set $Z$.
\end{dfn}

\begin{dfn} \label{dfn:70}
Let  $\bsym{M} = \{ M_k \}_{k \in \N}$ be an $\a$-adic system of 
$A$-modules. A {\em free resolution of $\bsym{M}$} is a  diagram of $\a$-adic 
systems
\[ \cdots \to \bsym{P}^{-2} \xar{\bsym{\d}^{-1}} \bsym{P}^{-1} 
\xar{\bsym{\d}^0} \bsym{P}^{0} 
\xar{\bsym{\eta}} \bsym{M} \to 0 , \]
such that each 
$\bsym{P}^i = \{ P_k^i \}_{k \in \N}$ is a free $\a$-adic system, and for every 
$k$ the diagram of $A$-modules 
\begin{equation} \label{eqn:115}
\cdots \to P_k^{-2} \xar{\d^{-1}_k} P_k^{-1} \xar{\d^0_k} P_k^{0} 
\xar{\eta_k} M_k \to 0 
\end{equation}
is an exact sequence, i.e.\ it is a free resolution of the $A_k$-module 
$M_k$. We denote this resolution by 
$\bsym{\eta} : \bsym{P} \to \bsym{M}$. 
\end{dfn}

In other words, a free resolution of $\bsym{M} = \{ M_k \}_{k \in \N}$ is a 
commutative diagram 
\[ \UseTips \xymatrix @C=5ex @R=4ex {
&
\vdots 
\ar[d]
&
\vdots 
\ar[d]
&
\vdots 
\ar[d]
&
\vdots 
\ar[d]
\\
\cdots 
\ar[r]
&
P_1^{-2}
\ar[r]
\ar[d]
&
P_1^{-1}
\ar[r]
\ar[d]
&
P_1^{0}
\ar[r]
\ar[d]
&
M_1^{}
\ar[r]
\ar[d]
&
0
\\
\cdots 
\ar[r]
&
P_0^{-2}
\ar[r]
&
P_0^{-1}
\ar[r]
&
P_0^{0}
\ar[r]
&
M_0^{}
\ar[r]
&
0
} \]
in $\cat{Mod} A$, where 
$P^{-i}_k \cong \opn{F}_{\mrm{fin}}(Z_i, A_k)$ for some sets $Z_i$,
the vertical homomorphisms come from the ring surjections 
$A_{k + 1} \to A_k$, and the rows are exact. 

For each $k \geq 0$ let $P_k$ be the complex of free $A_k$-modules 
\begin{equation} \label{eqn:320}
P_k := \bigl( \cdots \to P_k^{-2} \xar{\d_k^{-1}} P_k^{-1} 
\xar{\d_k^0} P_k^{0} \to 0  \to \cdots \bigr) 
\end{equation}
that occurs as part of the resolution (\ref{eqn:115}).
So there is a quasi-isomorphic of complexes $\eta_k : P_k \to M_k$.

\begin{lem} \label{lem:40}
Let $B$ be a commutative ring, with nilpotent ideal $\b$, and 
let $M$ be a $B$-module. Define the ring 
$\bar{B} := B / \b$ and the $\bar{B}$-module 
$\bar{M} := \bar{B} \ot_B M$. 
Suppose 
\begin{equation} \label{eqn:40}
\cdots \to \bar{P}^{-2} \xar{\bar{\d}^{-1}} \bar{P}^{-1} \xar{\bar{\d}^0}  
\bar{P}^{0} \xar{\bar{\eta}} \bar{M} \to 0 
\end{equation}
is a free resolution of the module $\bar{M}$ over the ring $\bar{B}$,
such that for each $i \geq 0$ we have 
$\bar{P}^{-i} = \opn{F}_{\mrm{fin}}(Z_i, \bar{B})$
for some set $Z_i$. 
Also suppose that 
$\opn{Tor}^{B}_i(\bar{B}, M) = 0$ for all $i > 0$. 

Then the resolution \tup{(\ref{eqn:40})} can be lifted to $B$. Namely, there is 
a free resolution
\begin{equation} \label{eqn:41}
\cdots \to P^{-2} \xar{\d^{-1}} P^{-1} \xar{\d^0}  
P^{0} \xar{\eta} M \to 0 
\end{equation}
of the $B$-module $M$, such that 
$P^{-i} = \opn{F}_{\mrm{fin}}(Z_i, B)$
for every $i$, and the sequence \tup{(\ref{eqn:40})} is gotten from the 
sequence \tup{(\ref{eqn:41})} by applying the functor 
$\bar{B} \ot_B -$. 
\end{lem}

\begin{proof}
The proof is by induction on $i = 0, 1, 2, \ldots$.

We start with $i = 0$. For each $z \in Z_0$, we choose an arbitrary element
$m_z \in M$ lifting the element $\bar{\eta}(\de_z) \in \bar{M}$. 
This determines a homomorphism of $B$-modules 
\[ \eta : \opn{F}_{\mrm{fin}}(Z_0, B) \to M . \]
The Nakayama Lemma (in its nilpotent version) says that $\eta$ is surjective. 

Now we take some $i \geq 0$, and we assume that a partial free resolution 
\begin{equation} \label{eqn:55}
P^{-i} \xar{\d^{-i + 1}} P^{-i + 1} \to \cdots \xar{\d^0} P^{0} 
\xar{\eta} M \to 0 
\end{equation}
was found, lifting 
\begin{equation} \label{eqn:56}
\bar{P}^{-i} \xar{\bar{\d}^{-i + 1}} \bar{P}^{-i + 1} \to \cdots 
\xar{\bar{\d}^0} \bar{P}^{0} \xar{\bar{\eta}} \bar{M} \to 0 \, ; 
\end{equation}
namely the sequence (\ref{eqn:56}) is gotten from (\ref{eqn:55}) by the 
operation $\bar{B} \ot_B -$. 
Define the modules $L^{-i}$ and $\bar{L}^{-i}$ to be the submodules of 
$P^{-i}$ and $\bar{P}^{-i}$ respectively that make the sequences 
\begin{equation} \label{eqn:57}
0 \to L^{-i} \to P^{-i} \xar{\d^{-i + 1}} P^{-i + 1} \to \cdots \xar{\d^0} 
P^{0} \xar{\eta} M \to 0 
\end{equation}
and 
\begin{equation} \label{eqn:204} 
0 \to \bar{L}^{-i} \to 
\bar{P}^{-i} \xar{\bar{\d}^{-i + 1}} \bar{P}^{-i + 1} \to \cdots 
\xar{\bar{\d}^0} \bar{P}^{0} \xar{\bar{\eta}} \bar{M} \to 0  
\end{equation}
exact. There is a homomorphism 
$\phi : L^{-i} \to \bar{L}^{-i}$ 
that is induced from the surjection $P^{-i} \to \bar{P}^{-i}$.

We are given that the $B$-modules $M, P^0, \ldots, P^{-i}$ satisfy 
$\opn{Tor}^{B}_j(\bar{B}, -) = 0$
for every $j > 0$. By the usual syzygy argument this is also satisfied by 
$L^{-i}$. Therefore the sequence gotten 
by applying $\bar{B} \ot_B - $ to the exact sequence (\ref{eqn:57}) remains 
exact, and thus it is isomorphic to the sequence (\ref{eqn:204}). 
This implies that the homomorphism 
$\bar{B} \ot_B L^{-i} \to \bar{L}^{-i}$ is bijective; and hence 
$\phi : L^{-i} \to \bar{L}^{-i}$ is surjective. 
As before, using the Nakayama Lemma, we can lift the given surjection 
\[ \bar{\d}^{-i} : \bar{P}^{-i - 1} = \opn{F}_{\mrm{fin}}(Z_{i + 1}, \bar{B}) 
\to \bar{L}^{-i} \]
to a surjection
\[ \d^{-i} : P^{-i - 1} = \opn{F}_{\mrm{fin}}(Z_{i + 1}, B) \to L^{-i} .  \]
\end{proof}

\begin{thm} \label{thm:164}
Let $A$ be a commutative ring, let $\a$ be an ideal in $A$, and let 
$\bsym{M} = \{ M_k \}_{k \in \N}$ be an $\a$-adic system of 
$A$-modules. 
The following conditions are equivalent\tup{:}
\begin{enumerate}
\rmitem{i} The $\a$-adic system $\bsym{M}$ admits a free resolution. 
\rmitem{ii} For every $k \geq 0$ and $i > 0$, the module 
$\opn{Tor}^{A_{k + 1}}_i(A_{k}, M_{k + 1})$ vanishes.  
\end{enumerate}
\end{thm}

\begin{proof}
(i) $\Rightarrow$ (ii): Assume $\bsym{M}$ admits a free 
resolution $\bsym{\eta} : \bsym{P} \to \bsym{M}$. We can calculate 
$\opn{Tor}^{A_{k + 1}}_i(A_{k}, M_{k + 1})$
using the free resolution 
$\eta_{k + 1} : P_{k + 1} \to M_{k + 1}$;
see formula (\ref{eqn:320}). Now the complex
$A_k \ot_{A_{k + 1}} P_{k + 1}$ 
is isomorphic to the complex $P_k$, and therefore 
\[ \opn{H}^{-i}(A_k \ot_{A_{k + 1}} P_{k + 1}) = 0 \]
for all $i > 0$. 

\medskip \noindent
(ii) $\Rightarrow$ (i): 
The proof is by induction on $k$. For $k = 0$ we choose a free resolution 
\[ \cdots \to P_0^{-2} \xar{\d^{-1}_0} P_0^{-1} \xar{\d^0_0} P_0^{0} 
\xar{\eta_0} M_0 \to 0 \]
of the $A_0$-module $M_0$,
where 
$P^{-i}_0 = \opn{F}_{\mrm{fin}}(Z_i, A_0)$ for some sets $Z_i$.
 
Now assume that for $k \geq 0$ we found a free resolution of the truncated 
$\a$-adic system $\{ M_{k'} \}_{ 0 \leq k' \leq k}$. Let 
$B := A_{k + 1}$ and $\bar{B} := A_{k}$. Consider the free resolution 
\[ \cdots \to P_k^{-2} \xar{\d^{-1}_k} P_k^{-1} \xar{\d^0_k} P_k^{0} 
\xar{\eta_k} M_k \to 0 \]
% \[ \cdots \to P_k^{-2} \xar{\d} P_k^{-1} \xar{\d} P_k^{0} 
% \xar{\eta} M_k \to 0 \]
over the ring $\bar{B}$. By Lemma \ref{lem:40} this resolution can be lifted to 
a free resolution of $M_{k + 1}$ over ring $B$, with the same indexing sets 
$Z_{0},  Z_{1}, \ldots$. 
\end{proof}

The theorem immediately implies:

\begin{cor} \label{cor:290}
If $\bsym{M}$ is a flat $\a$-adic system, then it has a free resolution. 
\end{cor}

%\cleardoublepage
\section{\texorpdfstring{$\a$}{a}-Adic Flatness}
\label{sec:adic-flatness} 

In this section we prove Theorem \ref{thm:168}, that 
characterizes adically flat modules.
As before, $A$ is a commutative ring and $\a$ is an ideal in it. 
We denote by $\cat{D}(A) = \cat{D}(\cat{Mod} A)$ the unbounded derived category 
of $A$. 

Let us first recall the notion of $\a$-torsion. An element $m$ in an $A$-module 
$M$ is called an {\em $\a$-torsion element} if $m$ is annihilated by some power 
of the ideal $\a$. The {\em $\a$-torsion submodule} $\Ga_{\a}(M)$ of $M$ is the 
submodule consisting of all $\a$-torsion elements. Thus 
\begin{equation} \label{eqn:210}
\Ga_{\a}(M) = \lim_{k \to} \, \opn{Hom}_A(A_k, M) \, ,
\end{equation}
where $A_k = A / \a^{k + 1}$ as before.
The module $M$ is called an {\em $\a$-torsion module} if $\Ga_{\a}(M) = M$.

It might be good to mention that if the ideal $\a$ is nilpotent, then all 
$A$-modules are both $\a$-torsion and $\a$-adically complete. Thus the 
discussion is only interesting when the ideal $\a$ is not nilpotent. 

\begin{dfn} \label{dfn:210}
Let $A$ be a commutative ring, let $\a$ be an ideal in $A$, and 
let $M$ be an $A$-module. We say that $M$ is {\em $\a$-adically flat} if 
$\opn{Tor}^{A}_i(N, M) = 0$
for every $\a$-torsion $A$-module $N$ and every $i > 0$. 
\end{dfn}

Here is a useful characterization of adic flatness. It holds in general -- no 
finiteness or completeness assumptions are needed on $A$, $\a$ or $M$. 

\begin{thm} \label{thm:168}
Let $A$ be a commutative ring, let $\a$ be an ideal in $A$, and 
let $M$ be an $A$-module. The following three conditions are equivalent\tup{:}
\begin{enumerate}
\rmitem{i} The $A$-module $M$ is $\a$-adically flat.

\rmitem{ii} For every $i > 0$ and $k \geq 0$ the module 
$\opn{Tor}^{A}_i(A_{k}, M)$ vanishes, and   
$A_{k} \ot_A M$ is a flat $A_k$-module.

\rmitem{iii} For every $i > 0$ the module 
$\opn{Tor}^{A}_i(A_{0}, M)$ vanishes, and   
$A_{0} \ot_A M$ is a flat $A_0$-module.
\end{enumerate}
\end{thm}

\begin{proof}
In terms of derived functors we have 
\[ \opn{Tor}^{A}_i(N, M) = \opn{H}^{-i}(N \ot_{A}^{\mrm{L}} M) \]
as $A$-modules.  Thus $M$ is $\a$-adically flat if and only if for every 
$\a$-torsion $A$-module $N$, and every $i > 0$, we have 
$\opn{H}^{-i}(N \ot_{A}^{\mrm{L}} M) = 0$.
This is equivalent to the condition that the canonical morphism 
\[ N \ot_A^{\mrm{L}} M \to N \ot_A M \]
in $\cat{D}(A)$ is an isomorphism.
Let us write $M_k := A_k \ot_A M$. 

\medskip \noindent 
(i) $\Rightarrow$ (ii): Since $A_k$ itself is an $\a$-torsion $A$-module, the 
modules $\opn{Tor}^{A}_i(A_{k}, M)$ vanish for $i > 0$. 
Thus there is an isomorphism 
$M_k \cong A_k \ot_A^{\mrm{L}} M$ in $\cat{D}(A_k)$.

It remains to prove that $M_k$ is flat over $A_k$. 
Take some $A_k$-module $N$. 
By the associativity of derived tensor products we have isomorphisms 
\[ N \ot_{A_k}^{\mrm{L}} M_k \cong 
N \ot_{A_k}^{\mrm{L}} A_k \ot_{A}^{\mrm{L}} M \cong
N \ot_{A}^{\mrm{L}} M  \]
in $\cat{D}(A_k)$. But $N$ is an $\a$-torsion module, so by our assumption we 
have \lb 
$\opn{H}^{-i}(N \ot_{A}^{\mrm{L}} M) = 0$ for all $i > 0$. 
Therefore 
$\opn{H}^{-i}(N \ot_{A_k}^{\mrm{L}} M_k) = 0$ for all $i > 0$.

\medskip \noindent 
(ii) $\Rightarrow$ (i): Since for every $i$ the functor 
$\opn{Tor}^{A}_i(-, M)$
commutes with arbitrary direct limits, it suffices to check the vanishing of
$\opn{Tor}^{A}_i(N, M) = 0$
when $N$ is a finitely generated $\a$-torsion $A$-module. But then $N$ is an 
$A_k$-module for some $k$. By assumption $M_k$ is flat over $A_k$, and 
$A_k \ot_A^{\mrm{L}} M \cong M_k$ in $\cat{D}(A)$. 
Then there are isomorphisms 
\[ N \ot_{A}^{\mrm{L}} M \cong 
N \ot_{A_k}^{\mrm{L}} A_k \ot_{A}^{\mrm{L}} M \cong 
N \ot_{A_k}^{\mrm{L}} M_k \cong N \ot_{A_k} M_k  \]
in $\cat{D}(A)$. It follows that 
$\opn{H}^{-i}(N \ot_{A} M) = 0$ for all $i > 0$. 

\medskip \noindent 
(ii) $\Rightarrow$ (iii): This is trivial. 

\medskip \noindent 
(iii) $\Rightarrow$ (ii): We will use induction on $k$: 
assuming that (ii) holds for $k$, we will prove it for $k + 1$. 
For $k = 0$ this is given. 

Let $N$ be an $A_{k + 1}$-module. Define $N' := \a \cd N \subseteq N$ and
$N'' := N / N'$.  We get an exact sequence 
\[ 0 \to N' \to N \to N'' \to 0 \]
of $A_{k + 1}$-modules, that we view as a distinguished triangle
\begin{equation} \label{eqn:211}
N' \to N \to N'' \xar{\vartriangle} 
\end{equation}
in $\cat{D}(A_{})$. 
Now the modules $N'$ and $N''$ are annihilated by $\a^{k + 1}$, so they are in 
fact $A_k$-modules. By our assumption 
$A_k \ot_{A}^{\mrm{L}} M \cong M_k$ in $\cat{D}(A_{k})$, and $M_k$ is a flat 
$A_k$-module. So we have isomorphisms
\begin{equation} \label{eqn:220}
N' \ot_{A}^{\mrm{L}} M \cong 
N' \ot_{A_k}^{\mrm{L}} A_k \ot_{A}^{\mrm{L}} M \cong 
N' \ot_{A_k}^{\mrm{L}} M_k \cong N' \ot_{A_k} M_k 
\end{equation}
in $\cat{D}(A_{k})$. The same for $N''$. Thus, upon application of the 
functor 
$- \ot_{A}^{\mrm{L}} M$
to the distinguished triangle (\ref{eqn:211}), we obtain the 
distinguished triangle
\begin{equation} \label{eqn:212}
N' \ot_{A_k} M_k \to N \ot_{A}^{\mrm{L}} M 
\to N'' \ot_{A_k} M_k \xar{\vartriangle}  
\end{equation}
in $\cat{D}(A)$.
We conclude that 
$\opn{H}^{-i}(N \ot_{A}^{\mrm{L}} M) = 0$
for all $i > 0$. In particular, taking $N = A_{k + 1}$, this tells us that 
$\opn{H}^{-i}(A_{k + 1} \ot_{A}^{\mrm{L}} M) = 0$
for all $i > 0$, and that \lb 
$A_{k + 1} \ot_{A}^{\mrm{L}} M \cong M_{k + 1}$
in  $\cat{D}(A_{k + 1})$.

It remains to prove that $M_{k + 1}$ is flat over $A_{k + 1}$. Again we take an 
arbitrary \lb $A_{k + 1}$-module $N$, and examine the distinguished triangle 
(\ref{eqn:211}), but now it is in the category $\cat{D}(A_{k + 1})$. Using the 
isomorphisms from (\ref{eqn:220}) we obtain these isomorphisms
\[ N' \ot_{A_{k + 1}}^{\mrm{L}} M_{k + 1} \cong
N' \ot_{A_{k + 1}}^{\mrm{L}} A_{k + 1} \ot_{A_{}}^{\mrm{L}} M_{} \cong
N' \ot_{A_{}}^{\mrm{L}} M_{} \cong N' \ot_{A_k} M_k \]
in $\cat{D}(A_{k + 1})$. The same for $N''$.
Therefore, when we apply the functor 
$- \ot_{A_{k + 1}}^{\mrm{L}} M_{k + 1}$
to (\ref{eqn:211}), we obtain the distinguished triangle 
\[ N' \ot_{A_k} M_k \to N \ot_{A_{k + 1}}^{\mrm{L}} M_{k + 1}
\to N'' \ot_{A_k} M_k \xar{\vartriangle} \]
$\cat{D}(A_{k + 1})$.
We see that 
$\opn{H}^{-i}(N \ot_{A_{k + 1}}^{\mrm{L}} M_{k + 1}) = 0$
for all $i > 0$. 
\end{proof}

\begin{rem} \label{rem:310}
Theorem \ref{thm:168} is similar to \cite[Lemma 051C]{SP}. The latter says that
if \lb $\opn{Tor}^{A}_1(A_0, M) = 0$ and $M_0$ is flat over $A_0$, then 
$M_k$ is flat over $A_k$ for all $k$, and 
$\opn{Tor}^{A}_1(N, M) = 0$
for every $\a$-torsion $A$-module $N$.
It seems that this condition on $M$ is strictly weaker than being $\a$-adically 
flat; but we did not look for an example demonstrating this.
\end{rem}

\begin{prop} \label{prop:250}
Let $A$ be a commutative ring, let $\a$ be an ideal in $A$, and 
let $M$ be an $\a$-adically flat $A$-module. Then the canonical morphism 
$\mrm{L} \La_{\a}(M) \to \La_{\a}(M)$
in $\cat{D}(A)$ is an isomorphism. 
\end{prop}

\begin{proof}
Let
\begin{equation} \label{eqn:280}
\cdots P^{-1} \to P^0 \xar{\eta} M \to 0 
\end{equation}
be a flat resolution of the module $M$. We can view this as a quasi-isomorphism 
of complexes $\eta : P \to M$.
Then the completion 
$\La_{\a}(\eta) : \La_{\a}(P) \to \La_{\a}(M)$
represents the canonical morphism 
$\mrm{L} \La_{\a}(M) \to \La_{\a}(M)$. 
We are going to prove that $\La_{\a}(\eta)$ is a quasi-isomorphism. 

For every $k \in \N$, applying $A_k \ot_A -$ to (\ref{eqn:280}), there is a 
sequence 
\begin{equation} \label{eqn:251}
 \cdots A_k \ot_A P^{-1} \to A_k \ot_A P^0 \xar{\eta_k} A_k \ot_A M \to 0 .
\end{equation}
The complex 
$A_k \ot_A P$ represents $A_k \ot_{A}^{\mrm{L}} M$.
Because $M$ is $\a$-adically flat, we see that 
$\opn{H}^{-i}(A_k \ot_A P) = 0$ for all $i > 0$. Therefore the sequence 
(\ref{eqn:251}) is exact. The Mittag-Leffler argument (see 
\cite[Proposition 1.12.4]{KS}) says 
that the inverse limit of the sequences (\ref{eqn:251}) is also exact.
But then 
$\La_{\a}(\eta) : \La_{\a}(P) \to \La_{\a}(M)$
is a quasi-isomorphism. 
\end{proof}

\begin{rem} \label{rem:500}
Let $\a$ be a finitely generated ideal in a commutative ring $A$.
Our $\a$-adically flat modules are very close to the {\em flat contramodules} 
that were introduced in \cite{Po1}, \cite{PoRo}. 
According to Positselski, an $A$-module $M$ is called a contramodule if it 
is $\a$-adically cohomologically complete in the sense of \cite{PSY1}, namely if 
the canonical morphism $M \to \mrm{L} \La_{\a}(M)$ in $\cat{D}(A)$ is an 
isomorphism. A contramodule $M$ is called a flat contramodule if each 
$M_k := A_k \ot_A M$ 
is a flat $A_k$-module. Thus, by our Theorem \ref{thm:262}(1) and Proposition 
\ref{prop:250}, if $M$ is an $\a$-adically complete $A$-module and the ideal 
$\a$ is weakly proregular, then $M$ is an $\a$-adically flat module iff it is a 
flat contramodule.
We thank L. Positselski for mentioning this to us. 
\end{rem}

\begin{rem} \label{rem:501}
Let $\a$ be a finitely generated ideal in a commutative ring $A$.
Denote by $\cat{Mod}_{\a\tup{-tor}} A$ the subcategory of $\a$-torsion 
modules. Let $M$ be an $\a$-adically complete $A$-module.
It is easy to see that if $M$ is $\a$-adically flat, then $M$ is {\em flat 
relative to the subcategory $\cat{Mod}_{\a\tup{-tor}} A$}, i.e.\ the functor 
\begin{equation} \label{eqn:500}
M \ot_A - : \cat{Mod}_{\a\tup{-tor}} A \to 
\cat{Mod}_{\a\tup{-tor}} A 
\end{equation}
is exact. The converse is true if the ideal $\a$ is weakly proregular,
by Theorem \ref{thm:262}(1). But what if $\a$ is not weakly proregular? 

Surprisingly, the converse does not hold without weak proregularity. 
Here is the precise statement, communicated to us privately by Positselski (and 
can be proved using ideas in his paper \cite{Po2}). 
Let $P$ be the $\a$-adic completion of the free $A$-module of countable rank; 
so in the notation of Section \ref{sec:decay} we have 
$P = \opn{F}_{\mrm{dec}}(\N, \what{A})$.
Since $A_k \ot_A P$ is a free $A_k$-module (cf.\ Proposition \ref{prop:500}), it 
follows that the $A$-module $P$ 
is flat relative to the subcategory $\cat{Mod}_{\a\tup{-tor}} A$. 
However, Positselski claims that {\em if the $A$-module $P$ is $\a$-adically 
flat, then the ideal $\a$ is weakly proregular}. 
\end{rem}

%\cleardoublepage
\section{Weakly Proregular Ideals}
\label{sec:WPR} 

We begin this section by recalling some facts about {\em weak proregularity}, 
copied from \cite{PSY1}. (Partial results were obtained earlier in \cite{LC}, 
\cite{AJL} and \cite{Scz}.)  
Suppose $\bsym{a} = (a_1, \ldots, a_n)$ is a finite sequence of elements in a 
commutative ring $A$. For each $k \geq 1$ we have the sequence of elements 
$\bsym{a}^k := (a_1^k, \ldots, a_n^k)$, 
and the corresponding Koszul complex $\opn{K}(A; \bsym{a}^k)$.
Recall that $\opn{K}(A; \bsym{a}^k)$ is a complex of finite rank free 
$A$-modules, concentrated in degrees $-n, \ldots, 0$. The Koszul complexes 
form an inverse system 
$\bigl\{ \opn{K}(A; \bsym{a}^k) \bigr\}_{k \in \N}$.
In cohomology we get an inverse system 
$\bigl\{ \opn{H}(\opn{K}(A; \bsym{a}^k)) \bigr\}_{k \in \N}$
of graded $A$-modules. 
(Actually, each $\opn{K}(A; \bsym{a}^k)$ is a commutative DG ring, 
$\bigl\{ \opn{K}(A; \bsym{a}^k) \bigr\}_{k \in \N}$
is an inverse system of DG rings, and 
$\bigl\{ \opn{H}(\opn{K}(A; \bsym{a}^k)) \bigr\}_{k \in \N}$
is an inverse system of graded rings.)
In degree $0$ we have 
\[ \bigl\{ \opn{H}^0(\opn{K}(A; \bsym{a}^k)) \bigr\}_{k \in \N}
= \{ A_k \}_{k \in \N} , \]
so the limit is $\what{A}$. The sequence $\bsym{a}$ is 
called {\em weakly proregular} if for every $i < 0$ the inverse system 
$\bigl\{ \opn{H}^i(\opn{K}(A; \bsym{a}^k)) \bigr\}_{k \in \N}$
is pro-zero; namely for every $k$ there is a $k' \geq k$ such that the 
homomorphism 
$\opn{H}^i(\opn{K}(A; \bsym{a}^{k'})) \to 
\opn{H}^i(\opn{K}(A; \bsym{a}^k))$
is zero. 

The ideal $\a$ in $A$ is called a {\em weakly proregular ideal} if it is 
generated by some weakly proregular sequence. It is known that when $A$ is 
noetherian, every finite sequence is weakly proregular; and hence every ideal 
in $A$ is weakly proregular. It is not hard to see that if $A \to B$ is a flat 
ring homomorphism, and $\a \subseteq A$ is a weakly proregular ideal, then the 
ideal $\b := B \cd \a \subseteq B$ is weakly proregular. It is known that if 
the ideal $\a$ is weakly proregular, then every finite sequence that generates 
$\a$ is weakly proregular. 

Given a finite sequence of elements $\bsym{a}$, the dual Koszul complex is
\[ \opn{K}^{\vee}(A; \bsym{a}^k) := 
\opn{Hom}_A \bigl( \opn{K}(A; \bsym{a}^k), A \bigr). \]
These complexes form a direct system, and in the  limit we have 
the {\em infinite dual Koszul complex}
\[ \opn{K}_{\infty}^{\vee}(A; \bsym{a}) := \lim_{k \to} \,
\opn{K}^{\vee}(A; \bsym{a}^k) . \]
This is a bounded complex of flat $A$-modules, concentrated in degrees 
$0, \ldots, n$. The complex 
$\opn{K}_{\infty}^{\vee}(A; \bsym{a})$ can be described quite easily.
For $n = 1$, and writing $a := a_1$, it is the complex 
\[ \opn{K}_{\infty}^{\vee}(A; a) = 
\bigl( \cdots \to 0 \to A \xar{\ \d \ } A[a^{-1}] \to 0 \to \cdots \bigr)  . \]
Here $A$ is in degree $0$, the localized ring $A[a^{-1}]$ is in degree $1$, and 
$\d$ is the ring homomorphism. For $n \geq 2$ we have 
\[ \opn{K}_{\infty}^{\vee}(A; \bsym{a}) = 
\opn{K}_{\infty}^{\vee}(A; a_1) \ot_A \cdots 
\ot_A \opn{K}_{\infty}^{\vee}(A; a_n) . \]

The infinite dual Koszul complex affords a more meaningful characterization
of weak proregularity. Let $\bsym{a}$ be a finite sequence in $A$, and let $\a$ 
the ideal it generates. It is known that $\bsym{a}$ is weakly 
proregular iff for every injective $A$-module $I$, the canonical homomorphism 
\[ \Ga_{\a}(I) \to \opn{K}_{\infty}^{\vee}(A; \bsym{a}) \ot_A I \]
is a quasi-isomorphism. 
This implies that when $\bsym{a}$ (and thus $\a$) is weakly 
proregular, for every complex $M \in \cat{D}(A)$ the is an isomorphism
\begin{equation} \label{eqn:250}
\opn{K}_{\infty}^{\vee}(A; \bsym{a}) \ot_A M \cong 
\mrm{R} \Ga_{\a}(M) 
\end{equation}
in $\cat{D}(A)$. Furthermore, this isomorphism is functorial in $M$. 

When the ideal $\a$ is weakly proregular, the derived functors
\[ \mrm{R} \Ga_{\a}, \mrm{L} \La_{\a} : \cat{D}(A) \to \cat{D}(A) \]
are adjoints to each other. Moreover, letting 
$\cat{D}(A)_{\a\tup{-tor}}$ and $\cat{D}(A)_{\a\tup{-com}}$
be the essential images of $\mrm{R} \Ga_{\a}$ and $\mrm{L} \La_{\a}$, 
respectively, the functor 
\[ \mrm{R} \Ga_{\a} : \cat{D}(A)_{\a\tup{-com}} \to 
\cat{D}(A)_{\a\tup{-tor}} \]
is an equivalence, with quasi-inverse $\mrm{L} \La_{\a}$. This is the {\em MGM 
equivalence} from \cite{PSY1}. 

\begin{lem} \label{lem:220}
Assume $\a$ is weakly proregular. Then for every $M, N \in \cat{D}(A)$ there is 
an isomorphism 
\[ \mrm{R} \Ga_{\a}(M) \ot_{A}^{\mrm{L}} \mrm{R} \Ga_{\a}(N) \cong 
\mrm{R} \Ga_{\a}(M \ot_{A}^{\mrm{L}} N) \]
in $\cat{D}(A)$. Moreover, this isomorphism is functorial in $M$ and $N$. 
\end{lem}

\begin{proof}
Let $\bsym{a}$ be a finite sequence that generates the ideal $\a$. Choose 
K-flat resolutions $P \to M$ and $Q \to N$. Then 
\[ \mrm{R} \Ga_{\a}(M) \ot_{A}^{\mrm{L}} \mrm{R} \Ga_{\a}(N) \cong 
\bigl( \opn{K}_{\infty}^{\vee}(A; \bsym{a}) \ot_A P \bigr) 
\ot_A \bigl( \opn{K}_{\infty}^{\vee}(A; \bsym{a}) \ot_A Q \bigr) \]
and 
\[ \mrm{R} \Ga_{\a}(M \ot_{A}^{\mrm{L}} N) \cong 
\opn{K}_{\infty}^{\vee}(A; \bsym{a}) \ot_A (P \ot_A Q) \]
in $\cat{D}(A)$. But according to \cite[Lemma 4.29]{PSY1} there is a 
quasi-isomorphism 
\[ \opn{K}_{\infty}^{\vee}(A; \bsym{a}) \ot_A
\opn{K}_{\infty}^{\vee}(A; \bsym{a}) \to 
\opn{K}_{\infty}^{\vee}(A; \bsym{a}) . \]
\end{proof}

\begin{thm} \label{thm:200}
Let $A$ be a commutative ring, let $\a$ be a weakly proregular ideal in $A$, 
and let $M$ be an $\a$-adically flat $A$-module, with $\a$-adic completion 
$\what{M}$. Then the $A$-module $\what{M}$ is $\a$-adically flat.
\end{thm}

Of course the completion $\what{M}$ is also $\a$-adically complete, because 
$\a$ is finitely generated (see \cite[Corollary 3.6]{Ye1}). 

\begin{proof}
Let $N$ be an $\a$-torsion $A$-module. We have the canonical homomorphism 
$\tau_M : M \to \what{M}$. Our plan is to prove that 
\begin{equation} \label{eqn:223}
\opn{id}_N \ot_{A}^{\mrm{L}} \, \tau_M : N \ot_{A}^{\mrm{L}} M \to 
N \ot_{A}^{\mrm{L}} \what{M}
\end{equation}
is an isomorphism in $\cat{D}(A)$. This will show that 
$N \ot_{A}^{\mrm{L}} \what{M}$ has no negative cohomology. 

Proposition \ref{prop:250} says that the canonical morphism 
$\mrm{L} \La_{\a}(M) \to \La_{\a}(M) = \what{M}$
in $\cat{D}(A)$ is an isomorphism. Therefore we can replace (\ref{eqn:223}) 
with the morphism 
\begin{equation} \label{eqn:224}
\opn{id}_N \ot_{A}^{\mrm{L}} \, \tau^{\mrm{L}}_M : N \ot_{A}^{\mrm{L}} M \to 
N \ot_{A}^{\mrm{L}} \mrm{L} \La_{\a}(M) 
\end{equation}
in $\cat{D}(A)$. See \cite[Proposition 3.7]{PSY1} regarding the morphism 
$\tau^{\mrm{L}}_M : M \to \mrm{L} \La_{\a}(M)$. 

The functoriality in Lemma \ref{lem:220} tells us that there is a commutative 
diagram 
\[ \UseTips \xymatrix @C=22ex @R=8ex {
\mrm{R} \Ga_{\a} (N \ot_{A}^{\mrm{L}} M)
\ar[r]^(0.44){ \mrm{R} \Ga_{\a} (\opn{id}_N \ot_{A}^{\mrm{L}} \, 
\tau^{\mrm{L}}_M) }
\ar[d]_{ \cong }
&
\mrm{R} \Ga_{\a} \bigl( N \ot_{A}^{\mrm{L}} \mrm{L} \La_{\a}(M) \bigr)
\ar[d]^{ \cong }
\\
\mrm{R} \Ga_{\a} (N)  \ot_{A}^{\mrm{L}}  \mrm{R} \Ga_{\a}(M)
\ar[r]^(0.45){ \mrm{R} \Ga_{\a}(\opn{id}_N) \, \ot_{A}^{\mrm{L}} \, 
\mrm{R} \Ga_{\a}(\tau^{\mrm{L}}_M)  }
&
\mrm{R} \Ga_{\a} (N) \ot_{A}^{\mrm{L}}
\mrm{R} \Ga_{\a} \bigl( \mrm{L} \La_{\a}(M) \bigr)
} \]
in $\cat{D}(A)$ with vertical isomorphisms. 
According to \cite[Lemma 7.6]{PSY1} the morphism 
\[ \mrm{R} \Ga_{\a}(\tau^{\mrm{L}}_M) : 
\mrm{R} \Ga_{\a}(M) \to 
\mrm{R} \Ga_{\a} \bigl( \mrm{L} \La_{\a}(M) \bigr) \]
is an isomorphism. We conclude that the morphism 
$\mrm{R} \Ga_{\a} (\opn{id}_N \ot_{A}^{\mrm{L}} \, \tau^{\mrm{L}}_M)$,
gotten from the application of the functor $\mrm{R} \Ga_{\a}$ to the morphism 
$\opn{id}_N \ot_{A}^{\mrm{L}} \, \tau^{\mrm{L}}_M$
in (\ref{eqn:224}), is an isomorphism. 

Finally, because the complexes $N \ot_{A}^{\mrm{L}} M$
and 
$N \ot_{A}^{\mrm{L}} \mrm{L} \La_{\a}(M)$
have torsion cohomology, they belong to the category 
$\cat{D}(A)_{\a\tup{-tor}}$. See \cite[Corollary 4.32]{PSY1}. But by 
\cite[Corollaries 4.30 and 4.31]{PSY1} the functors 
\[ \mrm{R} \Ga_{\a}, \, \opn{Id} : \cat{D}(A)_{\a\tup{-tor}} \to 
\cat{D}(A)_{\a\tup{-tor}} \]
are isomorphic. Therefore the morphism 
$\opn{id}_N \ot_{A}^{\mrm{L}} \, \tau^{\mrm{L}}_M$
is an isomorphism. 
\end{proof}

%\cleardoublepage
\section{Modules of Decaying Functions} \label{sec:decay}

In this section we prove Theorems \ref{thm:290} and \ref{thm:295}, that 
together constitute Theorem \ref{thm:262} from the Introduction. 
We assume that $A$ is a commutative ring, and $\a$ is a finitely generated 
ideal in it. As before, for each $k$ we write $A_k := A / \a^{k + 1}$. 

Let $M$ be an $\a$-adically complete $A$-module. Following \cite{Ye1}, a
function $f : Z \to M$ is called an {\em $\a$-adically decaying function} if 
for every $k \in \N$ the set 
\[ \{ z \in Z \mid f(z) \notin \a^{k + 1} \cd M \} \]
is finite. In terms of the canonical surjections
$\tau_{M, k} : M \to A_k \ot_A M$, we see that 
a function $f : Z \to M$ is decaying if and only if for every $k$ the composed 
function $\tau_{M, k} \circ f$ has finite support. 
The set of $\a$-adically decaying functions $f : Z \to M$ is denoted 
by $\opn{F}_{\mrm{dec}}(Z, M)$. It is an $A$-module, and there are inclusions 
of $A$-modules
\[ \opn{F}_{\mrm{fin}}(Z, M) \subseteq \opn{F}_{\mrm{dec}}(Z, M) 
\subseteq \opn{F}(Z, M) , \]
that are strict inclusions when the indexing set $Z$ is infinite and $M$ is not 
$\a$-torsion.

Let $M$ be some $A$-module, with $\a$-adic completion $\what{M}$. There is a 
homomorphism 
\[ \opn{F}_{\mrm{fin}}(Z, M) \to \opn{F}_{\mrm{dec}}(Z, \what{M}) , \quad 
f \mapsto \tau_M \circ f . \] 
It is known that $\opn{F}_{\mrm{dec}}(Z, \what{M})$ is
$\a$-adically complete, and it is uniquely isomorphic to the 
$\a$-adic completion of $\opn{F}_{\mrm{fin}}(Z, M)$
by the homomorphism above. See \cite[Corollary 2.9]{Ye1}.

Consider an $A$-module $M$, with $\a$-adic completion $\what{M}$.
For any $k \geq 0$ there is the module $M_k := A_k \ot_A M$,
and the canonical surjection $\pi_k : \what{M} \to M_k$.
Given a set $Z$ there is a homomorphism 
\[ \opn{F}_{\mrm{dec}}(Z, \pi_k) : \opn{F}_{\mrm{dec}}(Z, \what{M}) \to 
\opn{F}_{\mrm{fin}}(Z, M_k) . \]
It is easy to see that this is surjective -- just lift a finitely supported 
function $\bar{f} : Z \to M_k$ to a function $f : Z \to \what{M}$ with the same 
support. We obtain an induced surjective homomorphism 
\begin{equation} \label{eqn:505}
\phi_k :  A_k \ot_A \opn{F}_{\mrm{dec}}(Z, \what{M}) \to 
\opn{F}_{\mrm{fin}}(Z, M_k) . 
\end{equation}

\begin{prop} \label{prop:500}
Let $A$ be a commutative ring, $\a \subseteq A$ a finitely generated ideal, 
$M$ an $A$-module, $Z$ a set, and $k \in \N$. Then the homomorphism 
$\phi_k$ from \tup{(\ref{eqn:505})} is bijective. 
\end{prop}

\begin{proof}
Say $\a$ is generated by elements $a_1, \ldots, a_n$. Consider the polynomial 
ring $B := \Z[t_1, \ldots, t_n]$,
the ideal $\b := (t_1, \ldots, t_n)$,
and the ring homomorphism $B \to A$, $t_i \mapsto a_i$. 
Then $\a^{k + 1} = \b^{k + 1} \cd A$ for every $k$. By replacing $A$ with $B$, 
we can assume that $A$ is a noetherian ring. 

Define $L_k := \opn{Ker}(\phi_k)$. So for each $k$ there is an exact
sequence 
\begin{equation} \label{eqn:502}
0 \to L_k \to A_k \ot_A \opn{F}_{\mrm{dec}}(Z, \what{M}) \xar{\phi_k}
\opn{F}_{\mrm{fin}}(Z, M_k) \to 0 . 
\end{equation}
We get an inverse system $\{ L_k \}_{k \in \N}$, and its limit is 
$L :=  \lim_{\leftarrow k} \, L_k$. 
As shown in the proof of Theorem \ref{thm:230}, the homomorphisms 
$L_{k + 1} \to L_k$ are surjective. Passing to the inverse limit in 
(\ref{eqn:502}) we get the sequence
\[ 0 \to L \to 
\lim_{\leftarrow k} \, \bigl( A_k \ot_A \opn{F}_{\mrm{dec}}(Z, \what{M}) \bigr)
\xar{\phi} \lim_{\leftarrow k} \, \opn{F}_{\mrm{fin}}(Z, M_k) \to 0 , \]
and it is exact by the ML argument. Now according to \cite[Theorem 2.7]{Ye1} 
and \cite[Corollary 3.8]{Ye1}, the homomorphism $\phi$ is bijective (and the 
two limit modules are canonically isomorphic to $\opn{F}_{\mrm{dec}}(Z, 
\what{M})$). It follows that $L = 0$. But $L \to L_k$ is surjective, so 
$L_k = 0$ and $\phi_k$ is bijective. 
\end{proof}

An $A$-module $P$ is called an {\em $\a$-adically free module} if it is 
isomorphic to $\opn{F}_{\mrm{dec}}(Z, \what{A})$
for some set $Z$. The reason for the name is that the functor 
$Z \mapsto \opn{F}_{\mrm{dec}}(Z, \what{A})$
is left adjoint to the forgetful functor 
\[ \cat{Mod}_{\a\tup{-com}} A \to \cat{Set} . \]
See \cite[Corollary 2.6]{Ye1}.

According to \cite[Theorem 3.4(2)]{Ye1}, when $A$ is noetherian, every 
$\a$-adically free module is flat. In the weakly proregular case we have the 
next result. 

\begin{lem} \label{lem:310} 
Assume $\a$ is weakly proregular. If $P$ is an $\a$-adically free 
$A$-module module, then it is $\a$-adically flat.
\end{lem}

\begin{proof}
Choose an isomorphism 
$P \cong \opn{F}_{\mrm{dec}}(Z, \what{A})$. 
By \cite[Corollary 2.9]{Ye1}, $P$ is isomorphic to the completion of 
the free module $\opn{F}_{\mrm{fin}}(Z, A)$.  Since the latter is flat over 
$A$, Theorem \ref{thm:200} says that $P$ is $\a$-adically flat.
\end{proof}

Suppose  $\bsym{M} = \{ M_k \}_{k \in \N}$ is a flat $\a$-adic system of 
$A$-modules. By Corollary \ref{cor:290} there exists a free resolution 
\begin{equation} \label{eqn:291}
\cdots \to \bsym{P}^{-2} \xar{\bsym{\d}^{-1}} \bsym{P}^{-1} 
\xar{\bsym{\d}^0} \bsym{P}^{0} 
\xar{\bsym{\eta}} \bsym{M} \to 0 .
\end{equation}
of $\bsym{M}$. By definition, the resolution (\ref{eqn:291}) consists of a free 
resolution 
\begin{equation} \label{eqn:293}
\cdots \to P_k^{-2} \xar{\d_k^{-1}} P_k^{-1} 
\xar{\d_k^0} P_k^{0} \xar{\eta_k} M_k \to 0 
\end{equation}
of $M_k$ for every $k$, that fit into an inverse system. 
By deleting $M_k$ we get the complex $P_k$ from (\ref{eqn:320}). 

Let $\what{M}$ be the limit of the system $\bsym{M}$, and for each $i$ let 
$\what{P}^i$ be the limit of the system $\bsym{P}^i$. 
We get a sequence 
\begin{equation} \label{eqn:292}
\cdots \to \what{P}^{-2} \xar{\what{\d}^{-1}} \what{P}^{-1} 
\xar{\what{\d}^0} \what{P}^{0} \xar{\what{\eta}} \what{M} \to 0 
\end{equation}
of $A$-modules. The modules $\what{P}^{i}$ are $\a$-adically free, and 
$\what{M}$ is $\a$-adically complete (by Theorem \ref{thm:230}). 
Let us denote by $\what{P}$ the complex of $A$-modules 
\[ \what{P} :=
\bigl( \cdots \to \what{P}^{-2} \xar{\what{\d}^{-1}} \what{P}^{-1} 
\xar{\what{\d}^0} \what{P}^{0} \to 0  \to \cdots \bigr) . \]
Then 
$\what{\eta} : \what{P} \to \what{M}$
is a homomorphism of complexes. 

\begin{lem} \label{lem:290}
\mbox{}
\begin{enumerate}
\item The sequence \tup{(\ref{eqn:292})} is exact. 

\item For every $k \geq 0$, the sequence of $A_k$-modules gotten by applying 
$A_k \ot_A -$ to the sequence \tup{(\ref{eqn:292})} is isomorphic to the 
sequence \tup{(\ref{eqn:293})}. 
\end{enumerate}
\end{lem}

\begin{proof}
(1) The sequence (\ref{eqn:292}) is the inverse limit of the exact sequences 
(\ref{eqn:293}). By the Mittag-Leffler argument, the limit is an exact 
sequence. 

\medskip \noindent 
(2) This follows from Theorem \ref{thm:230}(2).
\end{proof}

\begin{thm} \label{thm:290}
Let $A$ be a commutative ring, let $\a$ be a weakly proregular ideal in $A$, 
and let  $\bsym{M} = \{ M_k \}_{k \in \N}$ be a flat $\a$-adic system of 
$A$-modules, with 
limit $\what{M} = \lim_{\leftarrow k} \, M_k$.
Then the $A$-module $\what{M}$ is $\a$-adically flat. 
\end{thm}

\begin{proof}
We have to prove that $\opn{Tor}_i^A(N, \what{M}) = 0$ for every $i > 0$ and 
every $\a$-torsion $A$-module $N$. Since $\opn{Tor}_i^A(-, \what{M})$ commutes 
with direct limits, we can assume that $N$ is a finitely generated module. Thus 
$N$ is an $A_k$-module for some $k$. 

Choose a free resolution 
$\bsym{\eta} : \bsym{P} \to \bsym{M}$. 
By Lemma \ref{lem:290}(1) we get an exact sequence (\ref{eqn:292}).
By Lemma \ref{lem:310} each $\what{P}^i$ is an $\a$-adically flat 
$A$-module. So $\what{P}$ is a bounded above complex of $A$-modules, that are 
left acyclic for the functor $N \ot_A -$. 
This says that the canonical morphism 
\[ N \ot_{A}^{\mrm{L}} \what{P} \to N \ot_{A} \what{P} \]
in $\cat{D}(A)$ is an isomorphism. On the other hand, since 
$\what{\eta} : \what{P} \to \what{M}$
is an isomorphism in $\cat{D}(A)$, we see that 
\[ N \ot_{A}^{\mrm{L}} \what{P} \cong N \ot_{A}^{\mrm{L}} \what{M} \]
in $\cat{D}(A)$. Therefore it suffices to prove that 
$\opn{H}^{-i}(N \ot_{A} \what{P}) = 0$ for all $i > 0$. 

Now 
\[ N \ot_{A} \what{P} \cong N \ot_{A_k} A_k \ot_{A}  \what{P} , \]
and by Lemma \ref{lem:290}(2) we know that 
$A_k \ot_{A}  \what{P} \cong P_k$
as complexes. So 
\[ N \ot_{A} \what{P} \cong N \ot_{A_k} P_k  \]
as complexes. Since $\eta_k : P_k \to M_k$ is a free resolution of $M_k$, we 
have 
\[ N \ot_{A_k} P_k \cong N \ot_{A_k}^{\mrm{L}} M_k  \]
in $\cat{D}(A_k)$. But $M_k$ is a flat $A_k$-module, and hence 
$\opn{H}^{-i}(N \ot_{A_k}^{\mrm{L}} M_k) = 0$
for all $i > 0$. 
\end{proof}

\begin{lem} \label{lem:295}
Assume $A$ is noetherian and $\a$-adically complete. Let $P$ be an 
$\a$-adically free $A$-module, and let $N$ be a finitely generated $A$-module.
For each $k \geq 0$ define 
$P_k := A_k \ot_A P$ and $N_k := A_k \ot_A N$. Then the canonical homomorphism 
\[ P \ot_A N \to \lim_{\leftarrow k} \, (P_k \ot_{A_k} N_k) \]
is bijective. 
\end{lem}

\begin{proof}
Fix some isomorphism 
$P \cong  \opn{F}_{\mrm{dec}}(Z, A)$. By \cite[Lemma 3.3]{Ye1} the 
canonical homomorphism 
\[ \opn{F}_{\mrm{dec}}(Z, A) \ot_A N \to \opn{F}_{\mrm{dec}}(Z, N) \]
is bijective. By \cite[Theorem 2.7]{Ye1}  the canonical homomorphism 
\[ \opn{F}_{\mrm{dec}}(Z, N) \to 
\lim_{\leftarrow k} \, \opn{F}_{\mrm{fin}}(Z, N_k) \]
is bijective.  And trivially the canonical homomorphism 
\[ \opn{F}_{\mrm{fin}}(Z, A_k) \ot_{A_k} N_k \to 
\opn{F}_{\mrm{fin}}(Z, N_k) \]
is bijective. Combining these isomorphisms we deduce that 
the canonical homomorphism 
\[ \opn{F}_{\mrm{dec}}(Z, A) \ot_A N \to \lim_{\leftarrow k} \,
\bigl( \opn{F}_{\mrm{fin}}(Z, A_k) \ot_{A_k} N_k \bigr) \]
is bijective. Finally, According to \cite[Theorem 3.4(1)]{Ye1} the canonical 
homomorphism 
\[ A_k \ot_A \opn{F}_{\mrm{dec}}(Z, A) \to \opn{F}_{\mrm{fin}}(Z, A_k) \]
is bijective.
\end{proof}

Here is the noetherian variant of Theorem \ref{thm:290}.

\begin{thm} \label{thm:295}
Let $A$ be a noetherian commutative ring, let $\a$ be an ideal in $A$, 
and let  $\bsym{M} = \{ M_k \}_{k \in \N}$ be a flat $\a$-adic system of 
$A$-modules, with limit $\what{M} = \lim_{\leftarrow k} \, M_k$.
Then the $A$-module $\what{M}$ is flat. 
\end{thm}

\begin{proof}
Let $\what{A}$ be the $\a$-adic completion of $A$, and let 
$\what{\a} := \a \cd \what{A} \subseteq \what{A}$. 
The completion $\what{M}$ is an $\what{A}$-module, 
and the action of $A$ on $\what{M}$ is through the flat ring homomorphism 
$A \to \what{A}$. So it suffices to prove that $\what{M}$ is a flat 
$\what{A}$-module. Next note that $\bsym{M}$ is a flat $\what{\a}$-adic system 
of $\what{A}$-modules. Taking these facts together, we see that we can replace 
$A$ with $\what{A}$ -- namely we can assume that $A$ is 
$\a$-adically complete. 
  
As observed in the proof of Theorem \ref{thm:290}, it is enough to prove 
that $\opn{H}^{i}(N \ot_{A}^{\mrm{L}} \what{M}) = 0$ for all $i < 0$
and all finitely generated $A$-modules $N$. 
Let $\bsym{\eta} : \bsym{P} \to \bsym{M}$ be a free resolution of 
$\bsym{M}$. By Lemma \ref{lem:290}(2) we have an exact sequence 
(\ref{eqn:292}), and by \cite[Theorem 3.4(2)]{Ye1} this is a flat resolution of 
$\what{M}$. Therefore 
\[ N \ot_{A}^{\mrm{L}} \what{M} \cong N \ot_A \what{P} \]
in $\cat{D}(A)$. 

For every $k \geq 0$ we have the complex of free $A_k$-modules $P_k$ from 
(\ref{eqn:320}), and the quasi-isomorphism 
$\eta_k : P_k \to M_k$. Because $M_k$ is a flat $A_k$-module, the sequence 
\[ \begin{aligned}
& \cdots \to N_k \ot_{A_k} P_k^{-2} \xar{\opn{id} \ot \, \d_k^{-1}} 
N_k \ot_{A_k} P_k^{-1} 
\\
& \qquad  \qquad \xar{\opn{id} \ot \, \d_k^0} 
N_k \ot_{A_k} P_k^{0} \xar{\opn{id} \ot \, \eta_k} N_k \ot_{A_k} M_k
\to 0 \to \cdots 
\end{aligned} \]
is also exact. The Mittag-Leffler argument tells us that in the limit we still 
have an exact sequence. Thus 
\[ \opn{H}^i \Bigl( \lim_{\leftarrow k} \, (N_k \ot_{A_k} P_k) \Bigr) = 0 \]
for all $i < 0$. Finally, by Lemma \ref{lem:295} we know that 
\[ N \ot_A \what{P} \cong 
\lim_{\leftarrow k} \, (N_k \ot_{A_k} P_k) \]
as complexes.
\end{proof}

\begin{cor} \label{cor:320}
Let $A$ be a noetherian commutative ring, let $\a$ be an ideal in $A$, 
and let $\what{M}$ be an $\a$-adically flat $\a$-adically complete $A$-module. 
Then $\what{M}$ is a flat $A$-module. 
\end{cor}

\begin{proof}
Let $\{ M_k \}_{k \in \N}$ be the $\a$-adic system induced by 
$\what{M}$, i.e.\ $M_k = A_k \ot_A \what{M}$. 
The completeness of $\what{M}$ says that $\what{M}$ is the limit of this system.
By Theorem \ref{thm:168} we know that this is a flat  $\a$-adic system.
Theorem \ref{thm:295} says that $\what{M}$ is flat. 
\end{proof}

%\cleardoublepage
\section{The Non-Noetherian Example} 
\label{sec:example}

In this section we prove Theorem \ref{thm:205}, that provides an example 
of an  $\a$-adically complete $\a$-adically flat $A$-module which is not flat. 
The necessary facts on K\"ahler differentials can be found in 
\cite[Section 26]{Ma}.

\begin{lem} \label{lem:300}
Let $\K$ be a field of characteristic $0$, and let $B := \K[[t]]$, the ring 
of power series in a variable $t$. Then the module of K\"ahler $1$-forms 
$\Om^1_{B / \K}$ is not finitely generated as a $B$-module. 
\end{lem}

\begin{proof}
Let $L := \K((t))$ be the field of Laurent series, which is the field of 
fractions of the ring $B$. By the localization property of K\"ahler 
differentials, there is an isomorphism 
\[ L \ot_{B} \Om^1_{B / \K} \cong \Om^1_{L / \K} \]
of $L$-modules. If $\Om^1_{B / \K}$ were finitely generated as $B$-module, 
then $\Om^1_{L / \K}$ would be finitely generated as $L$-module. We will prove 
that this is false.

Because we are in characteristic $0$, the rank of 
$\Om^1_{L / \K}$ as an $L$-module equals the transcendence degree of $L$ over 
$\K$. But it is known that this transcendence degree is infinite. When the 
base field $\K$ is countable this is an easy exercise; and the 
general case was proved in \cite[Lemma 1]{MS}. 
\end{proof}

\begin{thm} \label{thm:205}
Let $\K$ be a field of characteristic $0$, let $\K[[t_1]]$ and 
$\K[[t_2]]$ be the rings of power series in the variables $t_1$ and 
$t_2$, and let $A$ be the ring
\[ A := \K[[t_1]] \ot_{\K} \K[[t_2]] . \]
Let $\a$ be the ideal in $A$ generated by $t_1$ and $t_2$,
and let $\what{A}$ be the $\a$-adic completion of $A$. 
Then\tup{:} 
\begin{enumerate}
\item The ideal $\a$ is weakly proregular. 
\item The ring $A$ is not noetherian.
\item The ring $\what{A}$ is noetherian.
\item The ring $\what{A}$ is $\a$-adically flat over $A$. 
\item The ring $\what{A}$ is not flat over $A$. 
\end{enumerate}
\end{thm}

\begin{proof}
(1) This was proved in \cite[Example 4.35]{PSY1}. The idea is this. 
Let $A' := \K[t_1, t_2]$, and let $\a' \subseteq A'$ be the ideal generated by 
the variables. Since $A'$ is noetherian, the ideal $\a'$ is weakly proregular. 
Now $A' \to A$ is flat, and $\a = A \cd \a'$, and this easily implies that $\a$ 
is weakly proregular.

\medskip \noindent
(2) This too was proved in \cite[Example 4.35]{PSY1}. An enhancement of that 
proof is now used to prove item (5). Note that if $A$ were noetherian, then the 
ideal $I$ introduced below would have to be finitely generated. 

\medskip \noindent
(3) Since $\what{A} = \K[[t_1, t_2]]$, it is a noetherian ring. 

\medskip \noindent
(4) Because $\a$ is weakly proregular, this is a special case of Theorem 
\ref{thm:200}. 

\medskip \noindent 
(5) This is the new and challenging part of the theorem. 
To prove that the ring homomorphism $\tau : A \to \what{A}$ is not flat, 
we shall exhibit an exact sequence of $A$-modules that does not remain 
exact after base change to $\what{A}$.

Consider the surjective ring homomorphism 
\[ f : A = \K[[t_1]] \ot_{\K} \K[[t_2]] \to B = \K[[t]] \]
defined by $f(t_l) := t$ for $l = 1, 2$. 
Let $I := \opn{Ker}(f) \subseteq A$. 
There is an exact sequence of $A$-modules 
\begin{equation} \label{eqn:5}
 0 \to I \to A \xar{\, f \, } B \to 0 .
\end{equation}
Applying $\what{A} \ot_A -$ to this sequence we get a sequence of 
$\what{A}$-modules 
\begin{equation} \label{eqn:1}
0 \to \what{A} \ot_A I \to \what{A} \xar{} \what{A} \ot_A B \to 0 \, .
\end{equation}
We will prove that the sequence (\ref{eqn:1}) is not exact. 
(Actually $\what{A} \ot_A B \cong B$, but we won't need this fact.)

For the sake of contradiction, let us assume that the sequence (\ref{eqn:1}) is 
exact. Then \lb $\what{A} \ot_A I$ is an ideal in the noetherian 
ring $\what{A}$, and hence it is finitely generated as an $\what{A}$-module.

There is another surjective ring homomorphism
\[ \what{f} : \what{A} = \K[[t_1, t_2]] \to B = \K[[t]] \]
defined by $\what{f}(t_l) := t$ for $l = 1, 2$. 
Note that $\what{f} \circ \tau = f$ as ring homomorphisms $A \to B$. 

Define the $B$-module $N := I / I^2$.
We view $N$ as an $\what{A}$-module through the ring homomorphism
$\what{f} : \what{A} \to B$. Consider the surjective $\what{A}$-module 
homomorphism 
\[ \psi : \what{A} \ot_A I \to N \]
that extends the canonical $A$-module surjection $I \to I / I^2 = N$. 
Since $\what{A} \ot_A I$  is finitely generated as an $\what{A}$-module, it 
follows that $N$ is finitely generated as a $B$-module. 

Finally, the $B$-module $N$ is isomorphic to the module of K\"ahler 
differential $1$-forms $\Om^1_{B / \K}$. But we already proved, in Lemma 
\ref{lem:300}, that $\Om^1_{B / \K}$ is not a finitely generated $B$-module. 
\end{proof}

\begin{rem} \label{rem:205}
If $\K$ is a field of characteristic $p > 0$, then the $B$-module 
$\Om^1_{B / \K}$ is free of rank $1$, with basis the form $\d(t)$. Thus the 
proof of Theorem \ref{thm:205} does not work in this situation. We do not know 
whether the statements themselves are true...  
\end{rem}

%\cleardoublepage

\end{document}